\documentclass{article}
\usepackage{amsmath,amssymb,amsthm,mathrsfs,graphics,amscd,url}
\usepackage[all]{xy}
\title{Pathological quotient singularities in characteristic three which are not log canonical}
\author{Takahiro Yamamoto\thanks{
\noindent address: Department of Mathematics, Graduate School of Science, Osaka University (Toyonaka, Osaka 560-0043, JAPAN)\newline
e-mail: t-yamamoto@cr.math.sci.osaka-u.ac.jp
}}
\date{}
\makeindex

\renewcommand{\thefootnote}{*\arabic{footnote}}
\theoremstyle{definition}
\newtheorem{lem}{Lemma}[section]
\newtheorem{prop}[lem]{Proposition}
\newtheorem{thm}[lem]{Theorem}
\newtheorem{coro}[lem]{Corollary}
\def\SL{\mathrm{SL}(3,K)}
\def\GL{\mathrm{GL}}
\def\Spec{\mathrm{Spec}}
\def\Sing{\mathrm{Sing}}

\def\F{\mathbb F}
\def\A{\mathbb A}

\def\Q{\mathbb Q}
\def\Z{\mathbb Z}
\def\C{\mathbb C}
\def\O{\mathcal O}

\def\mrm{\mathrm}

\def\mat#1#2#3#4{\left[
\begin{array}{ccc}
#1&#2\\#3&#4
\end{array}
\right]}
\def\matrix#1#2#3#4#5#6#7#8#9{\left[
\begin{array}{ccc}
#1&#2&#3\\#4&#5&#6\\#7&#8&#9
\end{array}
\right]}
\def\vvec#1#2#3{\left[\begin{array}{c}#1\\#2\\#3\end{array}\right]}

\newcommand{\dprime}{{\prime\prime}}
\newcommand{\cent}{\mathrm{cent}}
\begin{document}
\maketitle
\renewcommand{\thefootnote}{\fnsymbol{footnote}}
\footnote[0]{I am grateful to Takehiko Yasuda for teaching me and giving many important comments. Without his help, this paper would not have been possible. This work was partially supported by JSPS KAKENHI Grant Number JP18H01112.}

\begin{abstract}
In characteristic zero, quotient singularities are log terminal. Moreover, we can check whether a quotient variety is canonical or not by using only the age of each element of the relevant finite group if the group does not have pseudo-reflections. In positive characteristic, a quotient variety is not log terminal, in general. In this paper, we give an example of a quotient variety which is not log terminal such that the quotient varieties associated to any proper subgroups is canonical. In particular, we cannot determine whether a given quotient singularity is canonical by looking at proper subgroups.
\end{abstract}
\tableofcontents
\section{Introduction}
Quotient singularities form one of the most basic classes of singularities. They behave well in characteristic zero. In characteristic zero, any quotient variety has log terminal singularities. Moreover, for a finite group $G\subset\GL(d,\C)$ without pseudo-reflection, if we want to know the singularity of $\C^d/G$, we can use Reid--Shepherd-Barron--Tai criterion \cite[Theorem 3.21]{Kol:1}. Namely, the following three conditions are equivalent:
\begin{itemize}
\item $\C^d/G$ is canonical (resp. terminal);
\item $\C^d/G$ is canonical (resp. terminal) for all cyclic subgroup of $C$;
\item $\mrm{age}(g)\geq1$(resp. $>1$) for any $g\in G$;
\end{itemize}
where $\mrm{age}(g)$ is the age of $g\in G$. Diagonalizing an element $g\in G\backslash\{1\}$, we write
\[g=\matrix{\lambda_l^{a_1}}{ }{ }{ }{\ddots}{ }{ }{ }{\lambda_l^{a_d}}\]
where $l$ is the order of $g$, $\lambda_l$ is the primitive $l$-th root $\exp(2\pi\sqrt{-1}/l)$ and all integers $a_i$ satisfy $0\leq a_i\leq l-1$. Then the age of $g$ is defined by
\[\mrm{age}(g)=\frac 1l\sum_{i=1}^da_i.\] 

In positive characteristic, if the given finite group is tame, then the quotient variety is again log terminal and we can use the Reid--Shepherd-Barron--Tai criterion. But if the group is wild, there exists a quotient variety which is not log terminal. In this paper, we give an even more pathological example.
\begin{thm}[Main Theorem,Theorem \ref{main}]
Let $C_3$ be the cyclic group of order three and $C_3^2$ be the product of two copies of it. Suppose that the group $C_3^2$ is embedded in $\SL$ and this embedding makes $C_3^2$ small (it means the image of embedding of \(C_3^2\) has no pseudo-reflection), where $K$ is algebraically closed field of characteristic three. Then the quotient variety $\A^3/C_3^2$ is not log canonical.
\end{thm}
The pathological point of this example is that the quotient variety associated to any proper subgroup of $C_3^2$ is canonical, but the quotient variety by $C^2_3$ is not log terminal.
If the cyclic group $C_3$ of order three acts on the affine space $\A^3$ over $K$ linearly and small, the quotient variety has a crepant resolution, and so it is canonical \cite{Yas:1}. Since all nontrivial proper subgroups of $C_3^2$ are isomorphic to $C_3$, the quotient varieties by proper subgroups are canonical. But the above theorem says that the quotient variety $\A^3/C_3^2$ is neither canonical nor log terminal. This is in contrast to the fact that, in characteristic zero, the discrepancy of a quotient variety is determined by the age of elements of the group.

We give the proof of the main theorem in the following way. Firstly, we give all the small actions of $C^2_3$. It is not determined uniquely, but we can parametrize them by $a\in K\backslash\F_3,b\in K$. Next, we give the explicit form of the quotient varieties $X$ for each action of $C^2_3$. We will find that the quotient varieties are classified in two types separated by whether $b=0$ or not about the parameter $b$ of the action. Finally, we construct the proper birational morphism $Y\rightarrow X$ with exceptional divisors whose discrepancy is smaller than $-1$, which shows the quotient varieties are not log canonical. This construction given by a few times blow up along the singular loci.

As an application of the main result, we give a criterion when a quotient variety associated to a small wild finite group is log terminal in dimension three and characteristic three. According to the criterion, we can judge the singularity of a quotient variety by seeing the order of the acting group.
\begin{coro}[Corollary \ref{cor:1}]\label{cor:2}
Let $G$ be a wild small finite group of $\GL(3,K)$ where $K$ is an algebraically closed field. We write $\# G=3^rn$ where $r,n$ are positive integer and $n$ is not divided by three.
\begin{itemize}
\item[(i)] If $r=1$ then $\A^3_K/G$ is log terminal.
\item[(ii)] If $r\geq 2$ then $\A^3_K/G$ is not log canonical, in particular, not log terminal.
\end{itemize}
\end{coro}
This follows from the claim that if $X^\prime$ is not log terminal and a morphism $\pi\colon X^\prime\rightarrow X$ is finite dominant and \'etale in codimension one then $X$ is not log terminal. We give the proof of the claim and the log canonical version of it at the same time. 

Lastly, we explain why we consider characteristic three in dimension three.
In dimension two, every wild action has pseudo-reflection. Thus the smallest example of small actions are three dimensional. According to \cite[Corollary 1.4]{Yas:2} and Theorem 6.4 of the previous paper,  the quotient variety \(\A^3/G\) with wild small action in characteristic \(p>3\) is not log canonical. On the other hand, by looking at the Jordan normal form, we see that any wild action in characteristic two and dimension three has pseudo reflection. So, in dimension three, characteristic three is the only possibility for interesting examples existing.

This paper is organized as follows. In section two, we list some preliminaries. In section three, we consider small actions of $C_3^2$. In section four, we write quotient varieties explicitly. In section five, we give the proof of main theorem by constructing some proper birational morphisms. In section six, we prove the Corollary \ref{cor:2} as an application of the main theorem.

A large part of this paper is based on the master thesis of the author.

The main result of this paper was announced in \cite{Kin} without proof.
\section{Preliminaries}
We fix an algebraically closed field $K$ of characteristic three. We denote the $n$-dimensional affine space over $K$ by $\A^n$, that is, the spectrum of the polynomial ring $K[x_1,\ldots,x_n]$. We regard $\A^3$ as a vector space on $K$ and denote its general linear group by $\GL(3,K)$. Let $I$ be the identity element of $\GL(3,K)$ and $\SL$ be the special linear subgroup. Note that the eigenvalues of an element of $\GL(3,K)$ whose order is three satisfy the equation $x^3=1$. So order three elements are contained in $\SL$.

When a group $G$ is embedded in $\mathrm{GL}(n,K)$ and $G$ acts on $n$-dimensional vector space $K^n$ via this embedding, we call $g\in G$ is pseudo-reflection if $(K^n)^g$ is $(n-1)$-dimensional subspace of $K^n$. When $G$ has no pseudo-reflection, we say that $G$ is \textit{small}. 

A variety means a separated integral scheme of finite type over $K$. Let $X$ be a $n$-dimensional normal variety. We define the \textit{canonical sheaf} $\omega_X$ of $X$ by $\omega_X=j_\ast(\bigwedge^n\Omega_{X_{reg}})$ where $X_{reg}=X\backslash\Sing(X)$ and $\Omega_{X_{reg}}$ is the $\O_{X_{reg}}$-module of differentials. We call the divisor class $K_X$ defined by $\omega_X$ the \textit{canonical divisor} of $X$. If $mK_X$ is a Cartier divisor for some $m\in\Z\backslash\{0\}$, we say $X$ is \textit{$\Q$-Gorenstein}.

Let $X$ be a normal variety. A prime divisor $E$ on  a normal variety $Y$ given with a birational morphism $f\colon Y\rightarrow X$ is called \textit{a divisor over $X$}. If there is a birational map $g\colon Y\dashrightarrow Y^\prime$, we denote the closure $\overline{g(E)}$ of the image of $E$ by $\cent_{Y^\prime}(E)$. Let $E^\prime$ be another divisor over $X$, which is a divisor on $Y^\prime$ with a birational morphism $f^\prime\colon Y^\prime\rightarrow X$. There is a natural birational map $(f^\prime)^{-1}\circ f\colon Y\rightarrow X\dashrightarrow Y^\prime$. If $\cent_{Y^\prime}(E)=E^\prime$, we say $E$ and $E^\prime$ are equivalent. This defines an equivalence relation on the set of divisors over $X$. The equivalence class of $E$ is determined by the associated valuation on $K(X)$. In other words, the set of equivalence classes of divisors over $X$ can be embedded in the set of valuations on $K(X)$.

Let $X$ be a $\Q$-Gorenstein normal variety. For a birational morphism $f\colon Y\rightarrow X$ where $Y$ is a normal variety, we write
\[
K_Y=f^\ast K_X+\sum_Da_DD
\]
where $D$ runs over the $f$-exceptional prime divisors on $Y$. We call $a_E$ the discrepancy of $E$. If $E$ is not an exceptional divisor, its discrepancy is zero.
Note that the discrepancy is determined by the valuation on $K(X)$ corresponding to the divisor $E$, equivalently by the equivalence class of $E$. If
\[
\inf\{a_E|E\text{ is a divisor over X}\}\geq-1\ (\text{resp.}>-1),
\]
we say $X$ has \textit{log canonical} (resp. \textit{log terminal}) \textit{singularities}, or simply $X$ is \textit{log canonical} (resp. \textit{log terminal}). 

\section{Actions of $C^2_3$ on affine space}
In this section, we describe all the small actions of $C^3_2$ on $\A^3$. All the not necessarily small embeddings in $\SL$ are given in section four in \cite{CSW:1}. We consider representations that $C_3^2$ is small and it is classified as type (1,1,1) in \cite{CSW:1}. From \cite[Proposition 4.1,Proposition 4.3]{CSW:1}, we get the following proposition.
\begin{prop} For any embedding to $\SL$ of $C_3^2$ for which $C_3^2$ is small, there exists $a\in K\backslash\F_3,b\in K$ such that $\sigma(U(a,b))$ is conjugate with the image of embedding where $U(a,b)$ is a subgroup of the additive group $(K^2,+)$ generated by $(1,0),(a,b)$ and $\sigma$ is the group homomorphism defined by
$$\sigma(c_1,c_2)=\matrix 1{-c_1}{c_1^2+c_2}01{c_1}001.$$
\end{prop}
\begin{proof}\label{prop:7.1}
For simplicity, we denote the image of the embedding of $C_3^2$ in $\SL$ by $C_3^2$ again. From \cite[Proposition 4.1]{CSW:1}, we get the finite subgroup $U$ of $(K^2,+)$ whose image by $\sigma$ is conjugate with $C_3^2$. Then $U$ is generated by two elements $(u,v),(u^\prime,v^\prime)$ and these are independent as the elements of $\F_3$-vector space since $U$ is isomorphic to $C_3^2$. If $u=0$ then $\sigma((u,v))\in\sigma(U)$ is a pseudo-reflection. This contradicts the assumption that $C_3^2$ is small. So we get $u\ne 0$. From \cite[Proposition 4.3]{CSW:1} , $U$ is conjugate to
$$U^\prime=\{u^{-1}(c_1,-u^{-2}vc_1+u^{-1}c_2)|(c_1,c_2)\in U\}$$
which is generated by $(1,0)$ and $(u^{-1}u^\prime,-u^{-2}vu^\prime+u^{-2}v^\prime)$. So we put $a=u^{-1}u^\prime,b=-u^{-2}vu^\prime+u^{-2}v^\prime$ and replace $U$ by $U^\prime$ , then we get the required form of $U$. Note that the above generators of $U$ are linearly independent over $\F_3$. If $a\in\F_3$, since $(1,0)$ and $(a,b)$ are independent as $\F_3$ vectors, we deduced $b\ne0$. Then $(0,b)=(a,b)-a(1,0)\in U$ and $\sigma((0,b))$ is pseudo-reflection. This contradicts the assumption that $\sigma(U)$ is small. So we get $a\not\in\F_3$. 
\end{proof}

\section{The quotient varieties associated to actions}
In this section, we compute the quotient varieties associated to the small actions of $C_3^2$. 
By \cite[Theorem 3.3,Theorem 6.3]{CSW:1}, we get concrete representations of quotient varieties.
\begin{prop}\label{invC2}
The quotient variety $X=\A^3/\sigma(U(a,b))$ is embedded in $\A^4_{x_1,x_2,x_3,x_4}$ as a hypersurface. We can represent it as
\begin{align*}
X=\left\{\begin{array}{cc}
V(x_2^9-x_3^2+x_1^9x_4+x_1^6H(x_1,x_2))&(b=0)\\
V(\alpha b^2x_2^5-b^4x_3^3-b^{10}x_1^6x_4+\alpha b^2x_1x_2^3x_3+x_1^3F(x_1,x_2,x_3))&(b\ne0)
\end{array}\right..
\end{align*}
where $\alpha=a^3-a$ and $H(s_1,s_2),\ F(s_1,s_2,s_3)$ are polynomials on $K$ defined by
\begin{align*}
H(s_1,s_2)&=(1+\alpha^2)x_2^6-\alpha^2x_1^2x_2^5+(1+\alpha^2)^2x_1^6x_2^2+\alpha^2(1+\alpha^2)x_1^8x_2^2+\alpha^4x_1^{10}x_2,\\
F(s_1,s_2,s_3)&=c_1s_2^4+c_2s_1s_2^2s_3+c_3s_1^2s_3^2+c_4s_1^3s_2^3+c_5s_1^4s_2s_3+c_6s_1^6s_2^3+c_7s_1^7s_3.
\end{align*}
for some $c_1,\dots,c_7\in K$.
\end{prop}
\begin{proof}
To compute invariant rings, we divided the groups $U(a,b)$ in two cases: one is when $b=0$, and the other is $b\ne 0$. 

\noindent{\bf Case: $b=0$}

We get $a\not\in\F_3$ from the condition given in Proposition.\ref{prop:7.1}. So,
\begin{align*}
\det\mat 1a00&=0,\\
\det\mat 1a1{a^3}&=a^3-a\ne0.
\end{align*}
Therefore we get
$$K[x,y,z]^{\sigma(U(a,0))}\cong K[x_1,x_2,x_3,x_4]/(x_2^9-x_3^2+x_1^9x_4+x_1^6H(x_1,x_2))$$
from \cite[Theorem 3.3]{CSW:1}. 

\noindent{\bf Case: $b\ne0$}

Now
\begin{align*}
\det\mat 1a0b&=b\ne0,\\
\det\mat 1a1{a^3}&=a^3-a\ne0.
\end{align*}
Computing the invariant ring according to \cite[Theorem 6.3]{CSW:1}, we get
$$K[x,y,z]^{\sigma(U(a,b))}\cong K[x_1,x_2,x_3,x_4]/(\alpha^3b^2x_2^5-b^4x_3^3-b^{10}x_1^6x_4+x_1^3F(x_1,x_2,x_3)).$$
\end{proof}
\section{Singularity}
In this section, we prove any quotient variety associated to a small action of $C_3^2$ is non log canonical. This proof is given by explicit calculations.
\begin{thm}\label{main}
When $C_3^2$ acts on $\A^3$ via an embedding $\SL$ for which $C_3^2$ is small, the quotient variety is not log canonical.
\end{thm}
\begin{proof}
Case $b=0$:

We can construct a proper birational morphism $\varphi\colon X_4\rightarrow X$ with a normal variety $X_4$ by four times blowing-up along the singular locus. We illustrate this situation by following diagram. 
$$\xymatrix@R=7pt@C=10pt{
Bl_{L_3}(W_3)\ar@{}[d]|\bigcup&Bl_{L_2}(W_2)\ar@{}[d]|\bigcup&Bl_{L_1}(W_1)\ar@{}[d]|\bigcup&Bl_{L_0}(W_0)\ar@{}[d]|\bigcup&\\
W_4\ar[r]^{\varphi_4}\ar@{}[d]|\bigcup&W_3\ar[r]^{\varphi_3}\ar@{}[d]|\bigcup&W_2\ar[r]^{\varphi_2}\ar@{}[d]|\bigcup&W_1\ar[r]^{\varphi_1}\ar@{}[d]|\bigcup&W_0\ar@{}[d]|\bigcup\\
X_4\ar[r]^{\varphi_4}\ar@{}[d]|\bigcup&X_3\ar[r]^{\varphi_3}\ar@{}[d]|\bigcup&X_2\ar[r]^{\varphi_2}\ar@{}[d]|\bigcup&X_1\ar[r]^{\varphi_1}\ar@{}[d]|\bigcup&X\ar@{}[d]|\bigcup\\
L_4&L_3&L_2&L_1&L_0
}$$
In this diagram, $L_i$ is the singular locus of $X_i$, $Bl_{L_i}(W_i),X_{i+1}$ are blow-ups of $W_i,X_i$ along $L_i$, $W_i$ is an open subvariety of $Bl_{L_i}(W_i)$ which contains $X_i$ for each $i=0,1,\ldots,5$. Morphisms $\varphi_i$ are the restrictions of the morphisms $Bl_{L_{i-1}}(W_{i-1})\rightarrow W_{i-1}$ determined by blow-ups. We represent these varieties explicitly by direct computation.

Firstly, $W_0=\A^4_{x_1,\ldots,x_4}$ and from Proposition \ref{invC2},
\begin{align*}
X&=V(x_2^9-x_3^2+x_1^9x_4+x_1^6H(x_1,x_2))\subset W_0,\\
L_0&=V(x_1,x_2,x_3).
\end{align*}
Blowing-up $W_0$ and $X$ along $L_0$, we get $W_1$ is covered by two open affine subvarieties $W_{1,t}=\A^4_{x_1,x_4,u_t,v_t},W_{1,u}=\A^4_{x_2,x_4,t_u,v_u}$ and
\begin{align*}
X_1\cap W_{1,t}&=V(x_1^7(u_t^9+x_4+x_1^3h_t)-v_t^2),\\
X_1\cap W_{1,u}&=V(x_2^7(1+t_u^9x_4+x_2^3h_u)-v_u^2),
\end{align*}
where $h_t=x_1^{-6}H(x_1,u_tx_1),\ h_u=x_2^{-6}t_u^6H(t_ux_2,x_2)$, which are polynomials in $x_1,u_t$ and $x_2,t_u$. The morphism $\varphi_1$ corresponds to the following homomorphisms of rings 
\begin{align*}
(\varphi_1|_{W_1,t})^\#:&K[x_1,x_2,x_3,x_4]\rightarrow K[x_1,x_4,u_t,v_t]\\
&[x_1,x_2,x_3,x_4]\mapsto[x_1,u_tx_1,v_tx_1,x_4]\\
 (\varphi_1|_{W_1,u})^\#:&K[x_1,x_2,x_3,x_4]\rightarrow K[x_2,x_4,t_u,v_u]\\
&[x_1,x_2,x_3,x_4]\mapsto[t_ux_2,x_2,v_ux_2,x_4]
\end{align*}
Computing the singular locus of $X_1$, we get
\begin{align*}
L_1\cap W_{1,t}=V(x_1,v_t),\ L_1\cap W_{1,u}=V(x_2,v_u).
\end{align*}

For $i=2,3,4$, $W_i,X_i$ are very similar to $W_1,X_1$. The varieties $W_i$ are covered by two open affine varieties $W_{i,t}=\A^4_{x_1,x_4,u_t,t_i},W_{i,u}=\A^4_{x_2,x_4,t_u,u_i}$. The $X_i$ are closed subvarieties of $W_i$ defined by
\begin{align*}
X_i\cap W_{i,t}&=V(x_1^{9-2i}(u_t^9+x_4+x_1^3h_t)-t_i^2),\\
X_i\cap W_{i,u}&=V(x_2^{9-2i}(1+t_u^9x_4+x_2^3h_u)-u_i^2).
\end{align*}
The $\varphi_i\ (i=2,3,4)$ corresponds to homomorphisms
\begin{align*}
(\varphi_i|_{W_{i,t}})^\#:&K[x_1,x_4,u_t,t_{i-1}]\rightarrow K[x_1,x_4,u_t,t_i]\\
&[x_1,x_4,u_t,t_{i-1}]\mapsto [x_1,x_4,u_t,t_ix_1]\\
(\varphi_i|_{W_{i,u}})^\#:&K[x_2,x_4,t_u,u_{i-1}]\rightarrow K[x_2,x_4,t_u,u_i]\\
&[x_2,x_4,t_u,u_{i-1}]\mapsto [x_2,x_4,t_u,u_ix_2]
\end{align*}
where $t_1=v_t$ and $u_1=v_u$.
For $i=2,3$, 
\begin{align*}
L_i\cap W_{i,t}=V(x_1,t_i),\ L_i\cap W_{i,u}=V(x_2,u_i),
\end{align*}
and for $i=4$,
$$L_4\cap W_{4,t}=V(x_1,u_t^9+x_4,t_4),\ L_4\cap W_{4,u}=V(x_2,1+t_u^9x_4,u_4).$$
Since $X_4$ is a local complete intersection, in particular, satisfies the \(S_2\) condition, and regular in codimension one, we get $X_4$ is normal by Serre's criterion.

Now, we compute the relation between the canonical divisor $K_{\widetilde X}$ and pull-buck $\varphi^\ast K_X$ of canonical divisor of $X$. Let $E_i$ be the exceptional divisor of $\varphi:W_i\rightarrow W_{i-1}$ and we also denote the strict transform of $E_i$ by $E_i$. Since $X_i$ are closed subvarieties of codimension one in $W_i$, we regard $X_i$ as divisors on $W_i$. Since the morphisms $\varphi_i$ are blow-ups along $L_i$, we get
$$K_{W_i}=\varphi_i^\ast K_{W_{i-1}}+\left\{\begin{array}{cc} 2E_i&(i=1)\\ E_i&(i=2,3,4)\end{array}\right.$$
and
$$\varphi^\ast_i X_{i-1}= X_i+2E_i$$
for $i=1,2,3,4$ where $X_0=X$.
By direct computation,
$$\varphi^\ast_i E_{i-1}=E_i$$
for $i=3,4$. Note that no prime divisor except \(E_i\) appears in \(\varphi_i^\ast E_{i-1}\) because we have restricted ourselves from \(Bl_{L_{i-1}}(W_{i-1})\) to an open subset \(W_i\). Combining these, since $\varphi=\varphi_4\circ\cdots\circ\varphi_1$, we get
\begin{align*}
K_{W_4}&=\varphi^\ast K_{W_0}+2(\varphi^\prime)^\ast E_1+3E_4,\\
X_4&=\varphi^\ast X-2(\varphi^\prime)^\ast E_1-6E_4.
\end{align*}
where $\varphi^\prime=\varphi_4\circ\varphi_3\circ\varphi_2$. Therefore
$$K_{W_4}+X_4=\varphi^\ast(K_{W_0}+X)-3E_4$$
and, by adjunction formula,
$$K_{X_4}=\varphi^\ast K_X-3E_4|_{X_4}.$$
So $X$ is not log canonical.

\noindent Case $b\ne0$:

From Proposition \ref{invC2}, we put
\begin{align*}
X&=\A^3/\sigma(U(a,b))\\
&=\Spec K[x_1,x_2,x_3,x_4]/(\alpha b^2x_2^5-b^4x_3^3-b^{10}x_1^6x_4+\alpha b^2x_1x_2^3x_3+x_1^3F(x_1,x_2,x_3)).
\end{align*}
Then $X$ is a singular variety and its singular locus is $L_0:=V(x_1,x_2,x_3)$. We construct a birational morphism $\varphi\colon X_2\rightarrow X$ with a normal variety $X_2$ by two times blowing-up along the singular locus. We illustrate this construction by the following diagram.
$$\xymatrix@R=7pt@C=10pt{
Bl_{L_1}(W_1)\ar@{}[d]|\bigcup&Bl_{L_0}(W_0)\ar@{}[d]|\bigcup&\\
W_2\ar[r]^{\varphi_2}\ar@{}[d]|\bigcup&W_1\ar[r]^{\varphi_1}\ar@{}[d]|\bigcup&W_0\ar@{}[d]|\bigcup\\
X_2\ar[r]^{\varphi_2}\ar@{}[d]|\bigcup&X_1\ar[r]^{\varphi_1}\ar@{}[d]|\bigcup&X\ar@{}[d]|\bigcup\\
L_2&L_1&L_0
}$$
We use the same symbols for varieties and some morphisms as used in the previous case. But its explicit forms may be different. We will describe these varieties.
The variety $W_1$ and the morphism $\varphi_1:W_1\rightarrow W_0$ are same as in previous case. The variety $X_1$ is defined by
\begin{align*}
X_1\cap W_{1,t}&=V(\alpha^3b^2u_t^5x_1^2-b^4v_t^3-b^{10}x_1^3x_4+\alpha b^2u_t^3v_tx_1^2+x_1^4f_1(x_1,u_t,v_t)),\\
X_1\cap W_{1,u}&=V(\alpha^3b^2x_2^2-b^4v_u^3-b^{10}t_u^6x_2^3x_4+\alpha b^2t_uv_ux_2^2+t_u^3x_2^4f_2(x_2,t_u,v_u)),
\end{align*}
in $W_1$ where $f_1(s_1,s_2,s_3)=s_1^{-4}F(s_1,s_1s_2,s_1s_3),f_2(s_1,s_2,s_3)=s_1^{-4}F(s_1s_2,s_1,s_1s_3)$. About the singular locus $L_1$, we get
\begin{align*}
L_1\cap W_{1,t}=V(x_1,v_t),\ L_1\cap W_{1,u}=V(x_2,v_u).\\
\end{align*}
The open variety $W_2$ of $Bl_{L_1}(W_1)$ is covered by three open varieties $W_{2,y}=\A^4_{x_1,x_4,u_t,y},W_{2,z}=\A^4_{x_4,u_t,v_t,z},W_{2,w}=\A^4_{x_4,t_u,v_u,w}$. The morphism $\varphi_2$ is given by
\begin{align*}
(\varphi_2|_{W_{2,y}})^\#:&K[x_1,x_4,u_t,v_t]\rightarrow K[x_1,x_4,u_t,y]\\
&[x_1,x_4,u_t,v_t]\mapsto[x_1,x_4,u_t,yx_1],\\
(\varphi_2|_{W_{2,z}})^\#:&K[x_1,x_4,u_t,v_t]\rightarrow K[x_4,u_t,v_t,z]\\
&[x_1x_4,u_t,v_t]\mapsto[zv_t,x_4,u_t,v_t],\\
(\varphi_2|_{W_{2,w}})^\#:&K[x_2,x_4,t_u,v_u]\rightarrow K[x_4,t_u,v_u,w]\\
&[x_2,x_4,t_u,v_u]\mapsto[wv_u,x_4,t_u,v_u].
\end{align*}
The variety $X_2$ is defined by
\begin{align*}
X_2\cap W_{2,y}&=V(\alpha^3b^2u_t^5-b^4y^3x_1-b^{10}x_1x_4+\alpha b^2u_t^3yx_1+x_1^2f_1(x_1,u_t,yx_1)),\\
X_2\cap W_{2,z}&=V(\alpha^3b^2u_t^5z^2-b^4v_t-b^{10}z^3x_4v_t+\alpha b^2u_t^3z^2v_t+z^4v_t^2f_1(zv_t,u_t,v_t)),\\
X_2\cap W_{2,u}&=V(\alpha^3b^2w^2-b^4v_u-b^{10}t_u^6w^3x_4v_u+\alpha b^2t_uw^2v_u+t_u^3w^4v_u^2f_2(wv_u,t_u,v_u)).
\end{align*}
Now, we can see that the singular locus $L_2$ of $X_2$ is contained in $W_{2,y}$. It has the form
\begin{align*}
L_2=V(x_1,u_t,y^3+b^6x_4)
\end{align*}
in $W_{2,y}$. Since $X_2$ is a local complete intersection and regular in codimension one, $X_2$ is normal by Serre's criterion.

Next, we compute the exceptional divisor of $\varphi=(\varphi_2\circ\varphi_1)\colon X_2\rightarrow X$. Since the morphisms $\varphi_1,\varphi_2$ are blow-up, we get
\begin{align*}
K_{W_1}=\varphi_1^\ast K_{W_0}+2E_1,\ 
K_{W_2}=\varphi_2^\ast K_{W_1}+E_2.
\end{align*}
where $E_1,E_2$ are the exceptional divisors of $\varphi_1\colon W_1\rightarrow W_0,\varphi_2\colon W_2\rightarrow W_1$. Regarding $X,X_1,X_2$ as divisor, we get
\begin{align*}
\varphi_1^\ast X=X_1+3E_1,\ 
\varphi_2^\ast X_1=X_2+2E_2
\end{align*}
from multiplicity of $X,X_1$ along $L_0,L_1$ respectively. By direct computation,
\begin{align*}
\varphi_2^\ast E_1=E_1+E_2.
\end{align*}
We put \(\varphi=\varphi_2\circ\varphi_1\). Then, by adjunction formula,
\begin{align*}
K_{X_2}&=(K_{W_2}+X_2)|_{X_2}\\
&=\varphi_2^\ast(K_{W_1}+X_1)|_{X_2}-E_2|_{X_2}\\
&=\varphi_2^\ast(\varphi_1^\ast(K_{W_0}+X)-E_1)|_{X_2}-E_2|_{X_2}\\
&=\varphi^\ast K_X-E_1|_{X_2}-2E_2|_{X_2}.
\end{align*}
As the coefficient \((-2)\) appears, $X$ is not log canonical.
\end{proof}
\section{Application}
In this section, we give other non log canonical quotient varieties using the main result. For this, we first prove some assertions.
\begin{lem}\label{lem:1}
Let $R\in\SL$ be a non pseudo-reflection element whose order is three. Then the centralizer $C_{\SL}(R)$ of $R$ is given by
\[
C_{\SL}(R)=\{aI+bR+cR^2|a,b,c\in K,\ a+b+c=1\}.
\]
In particular, $C_{\SL}(R)$ is abelian.
\end{lem}
\begin{proof}
Such an element $R$ is conjugate with the element
\[\matrix 010001100.\]
So we may replace $R$ by this matrix. When $A=[a_1,a_2,a_3]\in\SL$ where $a_1,a_2,a_3$ are vertical vectors,
\[AR=RA\ \Leftrightarrow\ Ra_1=a_3,\ Ra_2=a_1,\ Ra_3=a_2.\]
We put
\[a_3=\vvec cba,\]
then
\[A=\matrix abccabbca=aI+bR+cR^2.\]
Moreover, $\det A=a^3+b^3+c^3=(a+b+c)^3=1$ implies $a+b+c=1$. Hence we get
\[
C_{\SL}(R)\subset\{aI+bR+cR^2|a,b,c\in K,\ a+b+c=1\}.
\]
The other inclusion is obvious.
\end{proof}
\begin{lem}\label{lem:2}
Let $G$ be a small finite group of $\SL$ whose order is $3^r$. Then $G\cong C^r_3$
\end{lem}
\begin{proof}
Since the order of $G$ is a power of prime, the center $Z(G)$ of $G$ is not trivial. Take $R\in Z(G)$ whose order is three. Then $G\subset C_{\SL}(R)$. So Lemma \ref{lem:1} shows that $G$ is an abelian group and any element of $G$ has order three. Therefore, $G\cong C_3^r$ by the structure theorem of finitely generated abelian groups.
\end{proof}
\begin{lem}\label{lem:3}
Let $X^\prime$ be a variety and $\pi\colon X^\prime\rightarrow X$ be a finite dominant morphism. Then for any divisor $E^\prime$ over $X^\prime$, there exists the following diagram
\[\xymatrix{
Y^\prime\ar[r]^{f^\prime}\ar[d]_\rho&X^\prime\ar[d]^\pi\\
Y\ar[r]_f&X
}\]
where $f^\prime,f$ are birational morphisms with normal varieties $Y,Y^\prime$ and $\rho$ is a morphism satisfying the following conditions.
\begin{itemize}
\item The center of $E^\prime$ on $Y^\prime$ has codimension one.
\item The closure $\overline{\rho(\cent_{Y^\prime}(E^\prime))}$ is codimension one.
\end{itemize}
\end{lem}
\begin{proof}
Take a birational morphism $f_0^\prime\colon Y^\prime_0\rightarrow X^\prime$ such that $Y^\prime_0$ is normal variety and $E^\prime$ is a divisor on $Y^\prime_0$.
Let $\phi\colon X^{\prime\prime}\rightarrow X^\prime$ be the Galois closure of $\pi\colon X^\prime\rightarrow X$. Namely, the coordinate ring of $X^{\prime\prime}$ is the integral closure of $K[X]$ in a Galois closure of $K(X^\prime)/K(X)$. The morphism $\phi$ is defined by the inclusion $K[X^\prime]\rightarrow K[X^{\prime\prime}]$. We denote the Galois group of $L/K(X)$ by $G=\{g_1,\ldots,g_d\}$. It acts on $X^{\prime\prime}$ canonically.

We define a variety $Y^\dprime_0$ as the component of $Y^\prime_0\times_{X^\prime}X^{\prime\prime}$ such that the morphism $\varphi\colon Y^\dprime_0\rightarrow X^{\prime\prime}$ is dominant.

Moreover, we define a variety $Y^\dprime$ by
\[
Y^\dprime=((\cdots((Y^\dprime_{0,g_1}\times_{\sigma(g_1),X^\dprime,\sigma(g_2)}Y^\dprime_{0,g_2})\times_{X^\dprime,\sigma(g_3)}Y^\dprime_{0,g_3})\cdots)\times_{X^\dprime,\sigma(g_d)}Y^\dprime_{0,g_d})
\]
where $Y^\dprime_{0,g}$ is copy of $Y^\dprime_0$ with the morphism $\sigma(g)\colon Y^\dprime_0\xrightarrow{\varphi} X^\dprime\xrightarrow gX^\dprime$ for any $g\in G$. For $h\in G$, we define the action of $h$ on $Y^\dprime$ by the morphism induced by the morphisms $Y^\dprime_{0,g}\xrightarrow{id}Y^\dprime_{hg}$. Then $G$ acts on $Y^\dprime$. Furthermore, from the diagram
\[\xymatrix{
Y^\dprime_{0,g}\ar[r]^\varphi\ar[d]_{id}&X^\dprime\ar[r]^g\ar[d]_{id}&X^\dprime\ar[d]^h\\
Y^\dprime_{0,hg}\ar[r]_\varphi&X^\dprime\ar[r]_{hg}&X^\dprime
}\]
we get the morphism $Y^\dprime\rightarrow X^\dprime$ is $G$-equivariant. 
Let $H$ be a subgroup of $G$ which corresponds to the field extension $L/K(X^\prime)$. Then there exists natural morphisms $f^\prime\colon Y^\dprime/H\rightarrow X^\dprime/H=X^\prime,\ f\colon Y^\dprime/G\rightarrow X^\dprime/G=X$ since $Y^\dprime\rightarrow X^\dprime$ is $G$-equivariant. These are birational morphisms. The morphism $\rho\colon Y^\prime\rightarrow Y$ is defined naturally.
\[\xymatrix{
Y^\dprime\ar[d]_q\ar[r]&Y^\dprime_0\ar[d]\ar[r]&X^\dprime\ar[d]^\phi\\
Y^\prime\ar[d]_\rho\ar@/_8pt/[rr]_{f^\prime}&Y^\prime_0\ar[r]^{f^\prime_0}&X^\prime\ar[d]^\pi\\
Y\ar[rr]_f&&X
}\]
We consider a prime divisor $E^\dprime$ on $Y^\dprime$ contained in the pull-back of $E^\prime$ by $Y^\dprime\rightarrow Y^\dprime_0\rightarrow Y^\prime_0$. The push-forward $q_\ast E^\dprime$ of $E^\dprime$ by the natural morphism $q\colon Y^\dprime\rightarrow Y^\prime$ is a prime divisor on $Y^\prime$ since $q$ is finite. Now, let $v^\dprime,v^\prime$ be the valuations on $K(X^\dprime)=L,K(X^\prime)$ corresponding to $E,E^\prime$ respectively. By construction, the valuation $v^\dprime$ is an extension of $v^\prime$ on $K(X^\dprime)$. Since the push-forward of a prime divisor corresponds to restricting valuation, $q_\ast E^\dprime$ is equivalent to $E^\prime$ as a divisor over $X^\prime$. Moreover, since $\rho$ is also finite, $\rho$ preserve dimension.
\end{proof}
\begin{thm}\label{thm:1}
Let $X^\prime, X$ be a normal $\Q$-Gorenstein variety and $\pi\colon X^\prime\rightarrow X$ be a finite dominant morphism which is \'etale in codimension one. If $X^\prime$ is not log canonical (resp. not log terminal) then $X$ is not log canonical (resp. not log terminal).
\end{thm}
\begin{proof}
We prove only the assertion about log canonicity. The other assertion is similarly proved.

For any prime divisor $E^\prime$ over $X^\prime$ with discrepancy smaller than $-1$, we take a diagram in Lemma \ref{lem:3}.
\[\xymatrix{
Y^\prime\ar[r]^{f^\prime}\ar[d]_\rho&X^\prime\ar[d]^\pi\\
Y\ar[r]_f&X
}\]
We denote $\cent_{Y^\prime}(E^\prime)$ again by $E^\prime$ and $E=\overline{\rho(E^\prime)}$, which are prime divisors. We write
\begin{align*}
K_Y&=f^\ast K_X+aE+F,\\
K_{Y^\prime}&=\rho^\ast K_Y+bE^\prime+G^\prime\\
\rho^\ast E&=tE^\prime+H^\prime,
\end{align*}
where $F$ doesn't contain $E$ and $G^\prime,H^\prime$ don't contain $E^\prime$. Because $\pi$ is \'etale in codimension one, $K_{X^\prime}=\pi^\ast K_X$. Hence we get
\begin{align*}
K_{Y^\prime}&=\rho^\ast f^\ast K_X+(at+b)E^\prime+\rho^\ast F+G^\prime+aH^\prime\\
&=(f^\prime)^\ast\pi^\ast K_X+(at+b)E^\prime+\rho^\ast F+G^\prime+aH^\prime\\
&=(f^\prime)^\ast K_{X^\prime}+(at+b)E^\prime+\rho^\ast F+G^\prime+aH^\prime.
\end{align*}
By assumption, $at+b<-1$, and so $a<-\frac{b+1}t$. Since \(t\) is the ramification index of \(\rho\) along \(E^\prime\), by \cite[2.41]{Kol:1}, $b\geq t-1$. Therefore, we get $a<-1$.
\end{proof}

\begin{thm}\label{thm:2}
Let \(k\) be a field of characteristic \(p>0\) and \(X,X^\prime\) be normal \(\Q\)-Gorenstein varieties over \(k\). Suppose that \(\pi\colon X^\prime\rightarrow X\) is a finite dominant morphism of degree \(n\) which is \'etale in codimension one. If \(n\) is not divided by \(p\) and \(X^\prime\) is canonical, then \(X\) is log terminal.
\end{thm}
\begin{proof}For a birational map $f\colon Y\rightarrow X$ with a normal variety $Y$, we denote the normalization of the component of $X^\prime\times_{X}Y$ dominating $X^\prime$ by $Y^\prime$. We denote the composition of the normalization map and projections by $\rho\colon Y^\prime\rightarrow Y,\ g\colon Y^\prime\rightarrow X^\prime$. So we consider the following diagram.
\[\xymatrix{
Y^\prime\ar[r]^g\ar[d]_\rho& X^\prime\ar[d]^\pi\\
Y\ar[r]_f&X
}\]
Take an $f$-exceptional divisor $E$ and write $\rho^\ast E=\sum r_iE^\prime_i$ where $E^\prime_i$ are prime divisors. Then the equation
\[\sum_i r_ie_i=n\]
is held for some integer $e_i$. Since the right side of this equation is not divisible by \(p\), the one of $r_i$ is not divisible by p. So we denote a prime divisor with such coefficient by $E^\prime$ and its coefficient by $r$. We write
\begin{align*}
K_Y&=f^\ast K_X+aE+F,\\
K_{Y^\prime}&=\rho^\ast K_Y+bE^\prime+G^\prime\\
\rho^\ast E&=rE^\prime+H^\prime,
\end{align*}
where $F$ doesn't contain $E$ and $G^\prime,H^\prime$ don't contain $E^\prime$. Because $\pi$ is \'etale in codimension one, $K_{X^\prime}=\pi^\ast K_X$. Hence we get
\begin{align*}
K_{Y^\prime}&=\rho^\ast f^\ast K_X+(ar+b)E^\prime+\rho^\ast F+G^\prime+aH^\prime\\
&=g^\ast\pi^\ast K_X+(ar+b)E^\prime+\rho^\ast F+G^\prime+aH^\prime\\
&=g^\ast K_{X^\prime}+(ar+b)E^\prime+\rho^\ast F+G^\prime+aH^\prime.
\end{align*}
Since $X^\prime$ is canonical, $ar+b\geq 0$. Since \(r\) is the ramification index of \(\rho\) and \(\rho\) is tame, by \cite[2.41]{Kol:1}, $b=r-1$. Therefore, we get
\[a\geq -1+\frac 1r>-1.\]
Hence $X$ is log terminal.
\end{proof}

Now we give other non log canonical quotient varieties.
\begin{coro}\label{cor:1}
Let $G$ be a wild small finite subgroup of $\GL(3,K)$ and let $X$ be the quotient variety $\A^3/G$. We write $\# G=3^rn$ where $r,n\in\Z_{>0}$ and $n$ is not divided by three.
\begin{itemize}
\item[(i)] If $r=1$ then $X$ is log terminal.
\item[(ii)] If $r\geq 2$ then $X$ is not log canonical.
\end{itemize}\end{coro}
\begin{proof}
Firstly, we consider the case of $r=1$. By Lemma \ref{lem:2} and Sylow's theorem, $G$ has a subgroup $H$ isomorphic to $C_3$. Let $X^\prime$ be the quotient variety $\A^3/H$ and $\pi\colon X^\prime\rightarrow X$ be the canonical morphism, which is \' etale in codimension one. Note that the variety $X^\prime$ is canonical from \cite[Corollary 6.25]{Yas:1}. Since Theorem \ref{thm:2}, we get \(X\) is log terminal.

When $r\geq 2$, by Lemma \ref{lem:2} and Sylow's theorem, $G$ has a subgroup $H$ isomorphic to $C_3^2$. Then the quotient variety $\A^3/H$ is not log canonical and the natural morphism $\A^3/H\rightarrow\A^3/G$ is a finite dominant morphism which is \'etale in codimension one. By Theorem \ref{thm:1}, $\A^3/G$ is not log canonical.
\end{proof} 


\begin{thebibliography}{99}
\bibitem{Bat:1} V.V.Batyrev, \textit{Non-Archimedean integrals and stringy Euler numbers of log-terminal pairs}, J. Eur. Math. Soc. 1 (1999) 5-33
\bibitem{CW:1} H.E.A.E.Campbell, D.L.Wehlau, {\it Modular Invariant Theory}, Invariant Theory Algebr. Transform. Groups VIII (2010)
\bibitem{CSW:1} H.E.A.E.Campbell, R.J.Shank, D.L.Wehlau, \textit{Ring of invariants for modular representations of elementary abelian $p$-groups}, Invariant Theory Algebr. Transform. Groups Vol.18,No.1 (2013) pp.1-22
\bibitem{Kol:1} J.Koll\'ar, {\it Singularities of the minimal model program}, Cambridge Tracts in Math. (2013)
\bibitem{IR:1} Y.Ito, M.Reid, \textit{The McKay correspondence for finite subgroups of $\mathrm{SL}(3,\mathbf{C})$}, Higher-dimensional complex varieties (1996) 221-240
\bibitem{Yas:1} T.Yasuda, {\it The p-cyclic McKay correspondence via motivic integration}, Compos. Math.150(7):1125-1168, 2014.
\bibitem{Yas:2} T.Yasuda, {\it Discrepancies of p-cyclic quotient varieties}, Journal of Mathematical Sciences, the University of Tokyo, 26, 1-14(2019)
\bibitem{Kin} T.Yamamoto, {\it Pathological quotient singularities which are not log canonical in positive characteristic}, Kyoto University Research Information Repository, Proceedings of Kinosaki Algebraic Geometry Symposium 2018, 2018: 152-152 available at \underline{http://repository.kulib.kyoto-u.ac.jp/dspace/handle/2433/236420}.
\end{thebibliography}
\end{document}